\documentclass[11pt]{article}
\usepackage{amssymb}
\usepackage{amsmath}
\usepackage{color}
\usepackage{mathrsfs}
\usepackage{IEEEtrantools}

\newenvironment{proof}{\noindent {\bf Proof }}
{\hfill $\bullet$ \vspace{0.25cm}}

\newcommand{\dis}{\displaystyle}

\newtheorem{thm}{Theorem}
\newtheorem{prop}{\indent Proposition}

\newtheorem{lem}{\indent Lemma}

\newcommand{\mmmintone}[1]{{\dis{\int\kern -.36cm-}}_{\kern-.21cm\substack{#1}}\;\;}
\newcommand{\mmmintwo}[2]{{\dis{\int\kern -.43cm-}}_{\kern-.21cm\substack{#1}}^{\substack{#2}}\;\;}
\newcommand{\submint}{{\scriptstyle{\int\kern -.66em -}}}
\newcommand{\submintone}[1]{{\scriptstyle{\int\kern -.66em-}}_{\scriptscriptstyle{\kern-.21em\substack{#1}}}}
\newcommand{\fracmint}{{\textstyle{\int\kern -.88em -}}}
\newcommand{\fracmintone}[1]{{\textstyle{\int\kern -.88em
-}}_{\scriptscriptstyle{\kern-.21em\substack{#1}}}\;}

\title{A free boundary problem with non local interaction}

\author{Jimyeong Lee\footnote{  E-mail: ljm9667@gmail.com } \\Gran Sasso Science Institute, Via. F. Crispi 7, 67100 L'Aquila, Italy}

\date{\today}

\begin{document}

\maketitle

\begin{abstract}
We prove local existence for classical solutions of a free boundary problem which arises in one of the biological selection models proposed by Brunet and Derrida, \cite{BD97} and Durrett and Remenik, \cite{DR11}. The problem we consider describes the limit evolution of branching brownian particles on the line with death of the leftmost particle at each creation time as studied in \cite{DFPS2}. We use extensively results in \cite{Cannon} and \cite{Fasano}.
\end{abstract}
%
%

\vskip 1cm

%
%
%



\section{Introduction}
\label{sec:intro}
In \cite{BD97}
Brunet and Derrida have proposed several models
to study selection mechanisms in biological systems
which give rise to very interesting questions not
only in the applications to biology but also in the areas of
stochastic particle systems and PDE's with free boundaries.
This paper concerns mostly the last issue but it is worth, we think, to
give first a more general overview.

In the line of the Brunet-Derrida's proposal
Durrett and Remenik in \cite{DR11}
have introduced and studied a model of particles on $\mathbb{R}$ each of which, independently from the others, creates at rate 1 a new particle whose position is chosen randomly with probability $p(x,y)dy$, $p(x,y)=p(0,y-x)$, if $x$ is the position of the generating particle. Instantaneously after the creation the leftmost particle is deleted so that the total number of particles is constant.

The
biological interpretation is that the position of a particle is  ``its degree of fitness'', the rightmost particles are the most fitted.  The removal of the leftmost (and hence less fitted) particle gives rise to an improvement of the general fitness of the population and in fact Durrett and Remenik have proved the existence of traveling fronts moving with positive velocity.

The main difficulty in the analysis of the model is the apparently simple deleting mechanism of killing the leftmost particle.  In fact the notion of leftmost particle is highly non local: one needs to know the positions of \underline{all the particles} to determine which is the leftmost one. This is therefore a ``topological'' interaction which cannot be treated with the usual methods of interacting particle systems, it is the analogue in PDE's of free boundary problems in which the domain where the PDE's are defined is itself one of the unknowns,
see for instance the survey by Carinci, De Masi, Giardin\`a and Presutti, \cite{CDGP}, on topological interactions and their relation in the ``hydrodynamic limit'' with free boundary problems.

In the biological applications the size of the population is very large and therefore the main interest is in the analysis of the asymptotic behavior of the particle system in the continuum limit when $N$, (i.e.\  the total number of particles) diverges.
Under suitable assumptions on the initial datum
Durrett and Remenik have  proved that as $N\to\infty$ a limit density exists and it satisfies:
 \begin{equation}
 \label{intro.1}
 \frac{\partial}{\partial t} \rho(x,t)= \int_{X_t}^\infty dy \; p(y,x) \rho(y,t)dy, \quad \rho(x,0)=\rho_0(x)
 \end{equation}
where 
$X_t = \inf \{r:\rho(r,t) >0\}$.
Notice that the domain of integration on the right hand side of \eqref{intro.1} is
also an unknown since one needs to know  the whole function $\rho(x,t)$ to determine the value $X_t$ of ``the edge''.   

As it stands \eqref{intro.1} does not select $\rho(x,t)$ because we can give ``arbitrarily'' $X_t$ and still solve \eqref{intro.1}.
To get uniqueness we would need to know a priori $X_t$ which should be the limit position (as $N\to \infty$) of the leftmost particle in the system.  This is in itself an interesting issue  but apparently very difficult to address.  Durrett and Remenik have circumvented the difficulty
by using the other information coming from the particle system, namely that the total number of particles is conserved.  In the continuum limit where $N\to \infty$ the above is reflected into the condition that
 \begin{equation}
 \label{intro.2}
 \int_{X_t}^\infty dx \;   \rho(x,t) =1, \text{for all $t\ge 0$}
 \end{equation}
The pair  \eqref{intro.1}--\eqref{intro.2} is a ``free boundary problem'' but not in its more usual formulation where  \eqref{intro.1} is usually replaced by a parabolic diffusion equation and instead of \eqref{intro.2} there is a condition relating the velocity of the edge to the spatial derivative of the solution at the edge.  This is indeed what happens  in the classical Stefan problem, see for instance the survey by Fasano, \cite{Fasano}.

Under suitable assumptions on the initial datum $\rho_0$ and on the probability kernel $p(x,y)$
Durrett and Remenik have been able to prove that the pair \eqref{intro.1}--\eqref{intro.2} has a unique solution which is the limit density of the particles system.

An important ingredient in the proof is that $X_t$ is monotonically non decreasing, a feature that is clear at the particles level where in fact the position of the leftmost particle if it moves can only increase: it stays put when the new  particle is created to its left (because then this is the one which is deleted) while, in the other case, the previous second leftmost particle becomes the leftmost one.
Such a simplifying effect is not present in the next models we are going to discuss.

In  \cite{DFPS}  De Masi, Ferrari, Presutti and Soprano-Loto (in the sequel DFPS for brevity), have studied the so called N-BBM model, which is an acronym for $N$ branching Browniam motions.  The selection mechanism in the N-BBM model
is similar to the Durrett-Remenik one: once a new particle is created the leftmost one is deleted.  There are however two main differences: the particles move as independent Brownian motions and the new particle is created at exactly the same position of the generating one (the previous kernel $p(x,y)$ becomes a Dirac delta, $\delta(x-y)$).  Biologically this means that the individual fitness changes randomly in time and the duplicating processes are exact, the fitness of the son is exactly equal to that of the father.

Believing that the  Durrett-Remenik arguments extend to this case one would conjecture that the limit density $\rho(x,t)$ satisfies the equation
 \begin{equation}
 \label{intro.3}
 \frac{\partial}{\partial t} \rho(x,t)= \frac 12\frac{\partial^2}{\partial x^2} \rho(x,t)+ \rho(x,t), \quad x>X_t,\quad \; \rho(x,0)=\rho_0(x)
 \end{equation}
(where again  $X_t = \inf\{x:\rho(x,t)>0\}$).  \eqref{intro.3} is in fact obtained from \eqref{intro.1}  by adding on the right hand side
the Laplacian which takes into account the Brownian motion of the particles while the last term $\rho(x,t)$ is the right hand side of
 \eqref{intro.1} when $p(x,y)=\delta(x-y)$.

The free boundary problem \eqref{intro.3}--\eqref{intro.2}  is ``incomplete''
because even if $X_t$ is known yet \eqref{intro.3} does not have
a unique solution: we must also give the value of $\rho(x,t)$ at the edge $X_t$.

The natural choice would be to derive it as the limit particles density at the edge which is still not at all easy due to the poor control of the position of the leftmost particle.  However taking into account the regularizing effect of the heat diffusion one may suppose that
\begin{equation}
 \label{intro.4}
 \rho(X_t,t)= 0 \quad \text{at all times $t\ge 0$}
 \end{equation}
(notice  that in the Durrett-Remenik model \eqref{intro.4} does not hold, recall however that in \eqref{intro.1} there is no Laplacian !).

DFPS have proved (under suitable assumptions on the initial datum) that in the limit $N\to \infty$ the particle density has a limit $\rho(x,t)$ for any $t\ge 0$. It is also proved that  $\rho(x,t)$ satisfies \eqref{intro.3}--\eqref{intro.2}--\eqref{intro.4} if this has a ``regular'' solution.
As far as we know there is only a  ``local'' existence theorem
under suitable assumptions on the initial datum (as discussed in the next section) which therefore coincides with the limit density of the N-BBM system.
From \cite{DFPS} we know that $\rho(x,t)$ is well defined at all times, but it is not clear if at times larger than for local existence it is still a solution of \eqref{intro.3}--\eqref{intro.2}--\eqref{intro.4} at least in a ``weak sense''.

Notice that uniqueness in the local existence theorem follows from \cite{DFPS} as DFPS have shown that any ``smooth solution''  is necessarily equal to the limit density of the particles system and hence unique.

The question of traveling fronts  is of great interest: in
\cite{BBD17}
Berestycki, Brunet and Derrida have considered  \eqref{intro.3} complemented by conditions on the values of the solution and its derivative at the edge.  They were mainly interested in the precise asymptotics of the velocity of the front underlying connections with the Fisher-KPP type fronts, see also
\cite{GJ} where
Groisman and  Jonckheere discuss   front propagation and quasi-stationary distributions.

The analysis of the front before the limit $N\to \infty$ is also particularly interesting, see for instance the work of Maillard, \cite{maillard}  on its large fluctuations.

\medskip

We are mainly interested here in the existence of solutions for a free boundary problem introduced in \cite{DFPS2}.  The particles system is an extension of the N-BBM model obtained by making the branching mechanism non local as in the case considered by Durrett and Remenik. In  \cite{DFPS2} the conjectured evolution equation is in fact
 \begin{equation}
 \label{intro.5}
 \frac{\partial}{\partial t} \rho(x,t)= \frac 12\frac{\partial^2}{\partial x^2} \rho(x,t)+\int_{X_t}^\infty dy \; p(y,x) \rho(y,t)dy, \quad x>X_t,\quad \; \quad \rho(x,0)=\rho_0(x)
 \end{equation}
which is a combination of
\eqref{intro.1} (for the branching) and \eqref{intro.3} for the Brownian diffusion.  The results for the N-BBM model have been extended  in \cite{DFPS2} to this case, a limiting density $\rho(x,t)$ exists and it is uniquely defined, moreover if
there is a smooth solution of the free boundary problem \eqref{intro.5}--\eqref{intro.2}--\eqref{intro.4} then this is the limit particles density of the model.  As mentioned the proof of local existence of smooth solutions for \eqref{intro.5}--\eqref{intro.2}--\eqref{intro.4} is the main result in this paper, the precise statement is the following.

\vskip.5cm

\nopagebreak
{\bf Assumptions.}

\begin{itemize}

\item {\em On the initial datum.} We suppose that: $\rho_0(x)=0$ for $x\le 0$, it is in $C^3$ for $x\ge 0$ and it has compact support.  Moreover
     \begin{equation}
 \label{intro.6}
 \frac{d}{d x} \rho_0(x)\Big |_{x=0^+}= 2\int_0^\infty\int_0^\infty \, \rho_0(y) p(y,x)dydx >0
 \end{equation}

    \item {\em On the kernel $p(x,y)$.}   We suppose that: $p(x,y)=p(0,y-x)$, $p(0,x)$ is non negative with compact support, it is in $C^1$  and its integral is equal to 1 (i.e.\ $p(x,y)$ is a transition probability kernel).

\end{itemize}

\medskip

{\bf Remarks.}  By the first assumption   $X_0=0$: by translation invariance there is no loss in generality by fixing the edge initially at 0.  The regularity assumption on $\rho_0$ comes from the necessity of controlling the velocity of the edge which involves, as we will see, the second derivative of $\rho(x,t)$ with respect to $x$.  Finally the ``strange condition \eqref{intro.6}" is required to avoid initial layer problems as discussed in the next section.

\vskip .5cm

\begin{equation}
\left\{ \,
\begin{IEEEeqnarraybox}[][c]{l?s}
\IEEEstrut
\displaystyle \frac{\partial}{\partial t} \rho(x,t)= \frac 12\frac{\partial^2}{\partial x^2} \rho(x,t)+\int_{X_t}^\infty dy \; p(y,x) \rho(y,t)dy, \quad x>X_t,\\
\displaystyle \rho(x,0)=\rho_0(x), \quad x>0, \\
\displaystyle  \rho(X_t,t)= 0 \quad \text{at all times $t\ge 0$}, \\
 \displaystyle  \int_{X_t}^\infty dx \;   \rho(x,t) =1, \text{for all $t\ge 0$}.
\IEEEstrut
\end{IEEEeqnarraybox}
\right.
\label{intro.7}
\end{equation}

\begin{thm}
\label{thmintro.1}
Under the above assumptions there are $T>0$, $X_t$, $t\in[0,T]$, and $\rho(x,t)$, $x \ge X_t$, $t\in[0,T]$, such that:

\begin{itemize}

\item $X_0=0$, $X_t$ is differentiable and its derivative $V_t$ is H\"older continuous with exponent $1/2$.

\item  $\rho(x,t)$ is $C^{3,1}$ (three derivatives in $x$ and one in $t$) in the domain
${x> X_t, t>0}$.

\item  The pair $(X_t,\rho(x,t))$, ${x\ge X_t, t\in[0,T]}$, solves  the free boundary problem \eqref{intro.7}.

\end{itemize}

\end{thm}

\medskip
In the next section we  outline the strategy of the proof and  discuss what known in the literature. In Section 3 and 4, we state the main result and prove it. In Section \ref{sec.6} we give the proof of Theorem \ref{thmintro.1} and in the last section we discuss the extension to the other free boundary problems mentioned in this introduction.

\vskip2cm

\section{Strategy of proof}
\subsection{classical case}
We suppose tacitly hereafter that the initial position of the edge is $X_0=0$, then a
simple version of the classical Stefan free boundary problem is
    \begin{equation}
 \label{intro.7.0}
\rho_t = \frac 12 \rho_{xx}, \quad x \ge X_t, t\ge 0
 \end{equation}
    \begin{equation}
 \label{intro.7.1}
\rho(x,0)=\rho_0(x), \;\quad \rho(X_t,t) =  0
 \end{equation}
    \begin{equation}
 \label{intro.7.2}
\frac {dX_t}{dt} =- \rho_x(X_t,t)
 \end{equation}
where $X_t$ and $\rho(x,t)$ are the unknowns, to simplify notation space and time derivatives  are denoted hereafter by adding suffices.

The classical strategy for solving such a free boundary problem
is to fix a curve  $X_t$, solve \eqref{intro.7.0}-- \eqref{intro.7.1}, call $\rho(x,t)$ such a solution.  Define then a new curve  $\tilde{X}_t$ by setting its velocity $\tilde{V}_t= - \rho_x(X_t,t)$: this defines a map $\psi$: $X_t \to \tilde{X}_t$ and we look for a fixed point of $\psi$.  In the above case the best is to work on the compact space of uniformly Lipschitz curves $X_t$: one can then prove that for small $t$ the map $\psi$ is a contraction and then  the existence of a fixed point $X_t$ follows.  To conclude one must then show
that the solution of \eqref{intro.7.0}-- \eqref{intro.7.1} with $X_t$ the fixed point satisfies   \eqref{intro.7.2} as well.  See for instance \cite{Fasano}.

The N-BBM problem looks similar.  The differences are: (i)\; in \eqref{intro.7.0} there is an additional term on the right hand side, (ii)\;
there is the additional constraint that the mass is conserved, \eqref{intro.2}; (iii)\; we miss
the condition \eqref{intro.7.2}. (i) can be dealt with by changing variables: $\rho(x,t) \to w(x,t):=e^{-t} \rho(x,t)$.
 Conservation of mass can be written in differential form and a relation for the velocity of the front can be obtained by differentiating with respect to time \eqref{intro.4}.  One then obtains the system of equations:
      \begin{equation}
 \label{intro.7.0.0}
w_t = \frac 12 w_{xx}, \quad x \ge X_t, t\ge 0
 \end{equation}
    \begin{equation}
 \label{intro.7.1.0}
w(x,0)=\rho_0(x), \;\quad w(X_t,t) =  0
 \end{equation}
    \begin{equation}
 \label{intro.7.3}
w_x(X_t,t) = 2e^{-t}
 \end{equation}
  \begin{equation}
   \label{intro.7.4}
\frac {dX_t}{dt} =- \frac 14 e^{-t} w_{xx}(X_t,t)
 \end{equation}
Define next $u(x,t):= w_x(x,t)$ and ignore \eqref{intro.7.1.0}, we then get
     \begin{equation}
 \label{intro.7.0.0.0}
u_t = \frac 12 u_{xx}, \quad x \ge X_t, t\ge 0
 \end{equation}
    \begin{equation}
 \label{intro.7.3.0}
u(X_t,t) = 2e^{-t}
 \end{equation}
  \begin{equation}
   \label{intro.7.4}
\frac {dX_t}{dt} =- \frac 14 e^{-t} u_{x}(X_t,t)
 \end{equation}
The free boundary problem for $(X_t,u(x,t))$ looks like the classical Stefan problem
 \eqref{intro.7.0}--\eqref{intro.7.1}--\eqref{intro.7.2}: it requires an additional analysis which is contained in \cite{fasanoprimicerio} (for a more general system of equations), see also \cite{JMLEE2} where the above case is treated explicitly.

\subsection{non local case}
In the case we are mainly interested here there is a non local term  and this prevents us to use at least directly the above approach.  An alternative way to study the free boundary problems is to look at the evolution from the edge. This is often done at the particles level to study the shape and structure of the traveling waves independently of their location.
  We suppose tacitly hereafter that the initial position of the edge is $X_0=0$. The advantage of studying the free boundary problem in the frame where the edge is always at the origin is that by its very definition the spatial domain is fixed, it is no longer one of unknowns.  The difficulties however have not disappeared as in the evolution equations appears a drift term which depends on the velocity of the edge.  The natural setting from the problem requires now that the motion of the edge is $C^1$ (we will also require that the derivative $V_t$ is H\"older continuous with exponent $1/2$). More precisely we call $u(x,t) =  \rho_x(x,t)$ and then 
 change variables: $\rho(x,t)\to \rho(x-X_t,t)$, $u(x,t)\to u(x-X_t,t)$.  By an abuse of notation we denote by the same symbols $\rho$ and $u$ the new functions and we  get  the following system of equations:
 \begin{equation}
 \label{intro.15}
   \rho_t(x,t)= \frac 12 \rho_{xx}(x,t)+V_t  \rho_x(x,t)+\int_{0}^\infty dy \; p(y,x) \rho(y,t)dy, \quad x>0
 \end{equation}
 \begin{equation}
 \label{intro.16}
 \rho(x,0)=\rho_0(x),\; x \ge 0,\quad \rho(0,t)= 0,\; t\ge 0
 \end{equation}
  \begin{equation}
 \label{intro.17}
  u_t(x,t)= \frac 12  u_{xx}(x,t)+V_t u_x(x,t)+\int_{0}^\infty dy \; p(y,x) u(y,t)dy, \quad x>0
 \end{equation}
  \begin{equation}
 \label{intro.18}
 u(x,0) =u_0(x):=
 \frac {d\rho_0(x)}{dx}  \text{\;\;for $x\ge 0$ } 
  \end{equation}
    \begin{equation}
 \label{intro.19}
 u(0,t) = 2\int_{0}^\infty  \int_{0}^\infty dy\, \rho(y,t) p(y,x)dx
  \end{equation}
  We also have:
     \begin{equation}
 \label{intro.14}
 u(0,t)V_t  = - \frac 12  u_x(0,t) +\int_{0}^\infty\int_{0}^\infty dy \,   u(y,t)p(y,x)dydx
  \end{equation}
  which is obtained by differentiating \eqref{intro.4} with respect to time.
  
  In the way the above equations have been derived $u$ is the spatial derivative of $\rho$, but we will regard the  
  system \eqref{intro.15} to \eqref{intro.19} without imposing such relation.  Namely we fix a function $V_t$ which is H\"older continuous with exponent $1/2$; we then solve \eqref{intro.15}-\eqref{intro.16} and find $\rho(x,t)$.  We then solve \eqref{intro.17}-\eqref{intro.19} with $\rho(x,t)$ as determined above and thus get $u(x,t)$.  With such $\rho(x,t)$ and $u(x,t)$ we determine a new speed $V_t$ via
  \eqref{intro.14} and thus get an iterative scheme.  We will prove that all this can be done and the iterative scheme has a fixed point.  For such fixed point we re-establish the identity that $u$ is the spatial derivative of $\rho$ and then get a proof of Theorem \ref{thmintro.1}.
  
  The change of variables which fixes the position of the edge has been used in \cite{fasanoprimicerio2}.  Our approach is similar but we have extra difficulties for the presence of the non local term. Moreover, \cite{fasanoprimicerio2} relies on the result of \cite{LSU} which does not include our case since its initial and boundary conditions are stronger. As a further outcome of our analysis we prove Lemma \ref{lem.3} that is an improved version of estimate (4.24) of \cite{fasanoprimicerio2}. 
  
 \vskip 2cm

\section{Main results}
Let us denote by $K$ a Gaussian density function and by $G$ a Green function for a quarter plane as
\begin{eqnarray*}
\displaystyle K(x,t;\xi,\tau)=\frac{1}{\sqrt{2\pi(t-\tau)}}\exp\left\{-\frac{|x-\xi|^{2}}{2(t-\tau)}\right\},\ G(x,t;\xi,\tau)=K(x,t;\xi,\tau)-K(x,t;-\xi,\tau).
\end{eqnarray*}
From now on, we write positive constants as $\{c_i\}_{i\ge 1}$ and use the following facts extensively
\begin{eqnarray*}
\displaystyle \int_{-\infty}^{\infty}K(x,t;\xi,\tau)dx=1,\ \  \int_{-\infty}^{\infty}|K_x(x,t;\xi,\tau)|dx= \frac{\sqrt{2}}{\sqrt{\pi}}\frac{1}{\sqrt{t-\tau}}.
\end{eqnarray*}

From now on, $\displaystyle\lVert\ \cdot \ \rVert_{\infty}$ is $\displaystyle L^{\infty}$-norm in $\displaystyle D_{T}$, where $\displaystyle D_{T}=\{(x,t): 0<x,\ 0<t\leq T\}$.
 \begin{prop}
\label{prop.1}
Let $\displaystyle V\in C([0,T])$ where $T>0$. There is  a unique solution $\displaystyle\rho\in C(\overline{D_{T}})$ where $\displaystyle D_{T}=\{(x,t): 0<x,\ 0<t\leq T\}$ with $\displaystyle\rho_x\in C(\overline{D_{T}})$ and $\displaystyle\lVert\rho\rVert_{\infty}+\lVert\rho_x\rVert_{\infty}<\infty$ which satisfies \eqref{intro.15}-\eqref{intro.16}.
\end{prop}

\begin{proof}
Similarly as in Theorem 20.3.1 of \cite{Cannon}, let us define a mapping $\mathcal{F}: \mathscr{B_\eta}\rightarrow \mathscr{B_\eta}$, where $\displaystyle \mathscr{B_\eta}=\{\rho(x,t)\in C([0,\infty)\times[0,\eta]) : \rho_x\in C([0,\infty)\times[0,\eta]), \lVert\rho\rVert_{\infty}+\lVert\rho_x\rVert_{\infty}<\infty\}$ as
\begin{eqnarray}
\label{3241}
&&\displaystyle\mathcal{F}\rho(x,t):=\int_{0}^{\infty}G(x,t;\xi,0)\rho_0(\xi)d\xi \nonumber\\
&&\hspace{1cm}+\int_{0}^{t}\int_{0}^{\infty}G(x,t;\xi,\tau)\left[V_\tau\rho_\xi(\xi,\tau)+\int_{0}^{\infty}dy\rho(y,\tau)p(y,\xi)\right]d\xi d\tau.
\end{eqnarray}
Then $\displaystyle \mathscr{B_\eta}$ is a Banach space with the norm $\lVert \cdot \rVert_{\infty}+\lVert\frac{\partial}{\partial x}(\cdot)\rVert_{\infty}$ and we have 
\begin{eqnarray}
\displaystyle\lVert \mathcal{F}\rho_1-\mathcal{F}\rho_2 \rVert_{\infty}\leq 2\eta\left[\lVert V\rVert_{\infty}\lVert {\rho_{1}}_{x}- {\rho_{2}}_{x} \rVert_{\infty}+\lVert \rho_{1}- \rho_{2} \rVert_{\infty}\right],
\end{eqnarray}
and 
\begin{eqnarray}
\displaystyle \left\lVert \frac{\partial\mathcal{F}\rho_1}{\partial x}-\frac{\partial\mathcal{F}\rho_2}{\partial x} \right\rVert_{\infty}\leq c_1\sqrt{\eta}\left[\lVert V\rVert_{\infty}\lVert {\rho_{1}}_{x}- {\rho_{2}}_{x} \rVert_{\infty}+\lVert \rho_{1}- \rho_{2} \rVert_{\infty}\right].
\end{eqnarray}
Thus for all sufficiently small $\eta>0$, we get $\mathcal{F}$ is a contraction mapping so that there is a unique fixed point $\rho^\eta$. To extend $\eta$ to $T$, if we define $\mathcal {H}$ as 
\begin{eqnarray*}
&&\displaystyle\mathcal{H}\rho(x,t):=\int_{0}^{\infty}G(x,t;\xi,0)\rho_0(\xi)d\xi\\
&&\hspace{2cm}+\int_{0}^{\eta}\int_{0}^{\infty}G(x,t;\xi,\tau)\left[V_\tau{\rho^{\eta}}_\xi(\xi,\tau)+\int_{0}^{\infty}dy\rho^{\eta}(y,\tau)p(y,\xi)\right]d\xi d\tau\\
&&\hspace{2cm}+\int_{\eta}^{t}\int_{0}^{\infty}G(x,t;\xi,\tau)\left[V_\tau\rho_\xi(\xi,\tau)+\int_{0}^{\infty}dy\rho(y,\tau)p(y,\xi)\right]d\xi d\tau.
\end{eqnarray*}
Then we have
\begin{eqnarray}
\displaystyle\lVert \mathcal{H}\rho_1-\mathcal{H}\rho_2 \rVert_{\infty}\leq 2(t-\eta)\left[\lVert V\rVert_{\infty}\lVert {\rho_{1}}_{x}- {\rho_{2}}_{x} \rVert_{\infty}+\lVert \rho_{1}- \rho_{2} \rVert_{\infty}\right],
\end{eqnarray}
and 
\begin{eqnarray}
\displaystyle \left\lVert \frac{\partial\mathcal{H}\rho_1}{\partial x}-\frac{\partial\mathcal{H}\rho_2}{\partial x} \right\rVert_{\infty}\leq c_1\sqrt{t-\eta}\left[\lVert V\rVert_{\infty}\lVert {\rho_{1}}_{x}- {\rho_{2}}_{x} \rVert_{\infty}+\lVert \rho_{1}- \rho_{2} \rVert_{\infty}\right]
\end{eqnarray}
such that we have $\mathcal{H}$ is a contraction mapping from $\displaystyle \{\rho(x,t)\in C([0,\infty)\times[\eta, 2\eta]) : \rho_x\in C([0,\infty)\times[\eta,2\eta]), \lVert\rho\rVert_{\infty}+\lVert\rho_x\rVert_{\infty}<\infty\}$ to itself. Thus we can extend $\eta$ to $2\eta$, inductively also to $T$ and this completes the proof.
\end{proof}

We will prove the existence of a classical solution of the FBP \eqref{intro.17}-\eqref{intro.14}  in $[0,T]$ in the next sections:
\begin{thm}
\label{3424567}
 There is $T>0$ and a pair $(V,u)$ which satisfies the FBP \eqref{intro.17}-\eqref{intro.14} in $[0,T]$ with:   $\displaystyle V\in C([0,T])$ and $u\in C(\overline{D_{T}})\cap C^{2,1}(D_{T})$, where $\displaystyle D_{T}=\{(x,t): 0<x,\ 0<t\leq T\}$.
\end{thm}

\medskip

We did not find a proof of Theorem \ref{3424567} in the existing literature, see for instance \cite{BH}, \cite{CV} and references therein.  Our proof exploits the one dimensionality of the problem and uses extensively the Cannon estimates, \cite{Cannon}, following the strategy proposed by Fasano in \cite{Fasano}. Then we will prove Theorem \ref{thmintro.1} as a corollary  of Theorem \ref{3424567}.

\vskip 2cm

\section{Proof of  Theorem \ref{3424567}}
\label{sec.4}
Theorem \ref{3424567} is proved at the end of the section. The idea is to reduce the analysis of the FBP \eqref{intro.17}-\eqref{intro.14} to a fixed point problem:

\begin{itemize}

\item Take a curve $V_t, t\in[0,T]$, and find $u$ such that $(V,u)$ solves \eqref{intro.17}-\eqref{intro.19}.

\item Construct a new curve $\displaystyle Q[V](t)=\frac{-\frac{1}{2}u_x(0,t)+\int_{0}^{\infty}\int_{0}^{\infty}u(y,t)p(y,x)dydx}{u(0,t)}$ for $0\leq t\leq T$

\item Find $V$ so that $Q[V]=V$ and prove that the corresponding pair $(V,u)$
solves the FBP \eqref{intro.17}-\eqref{intro.14}.

\end{itemize}

 The first task is to prove existence and smoothness of $u$, that we do in this section using   the   lemmas below, see \cite{Cannon}.
   
%
%
%
%
%
%

\medskip

\begin{prop}
\label{prop.2}
Let $\displaystyle V\in C([0,T])$ where $T>0$. There is  a unique solution $\displaystyle u\in C(\overline{D_T})$ where $\displaystyle D_T=\{(x,t): 0<x,\ 0<t\leq T\}$ with $\displaystyle  u_x\in C(\overline{D_T})$ and $\displaystyle\lVert u\rVert_{\infty}+\lVert u_x\rVert_{\infty}<\infty$ such that $(V,u)$ satisfies \eqref{intro.17}-\eqref{intro.19}.
\end{prop}
\begin{proof}
Let $\rho$ as in Proposition \ref{prop.1} and let us define a mapping $\mathcal{F}:\mathscr{B_\eta}\rightarrow \mathscr{B_\eta}$ where $\displaystyle \mathscr{B_\eta}=\{u(x,t)\in C([0,\infty)\times[0,\eta]) : u_x\in C([0,\infty)\times[0,\eta]), \lVert u\rVert_{\infty}+\lVert u_x\rVert_{\infty}<\infty\}$ as
\begin{eqnarray*}
\label{12345678}
&&\mathcal{F}u(x,t):=-\int_{0}^{t}K_x(x,t;0,\tau)g(\tau)d\tau+\int_{0}^{\infty}G(x,t;\xi,0)\varphi(\xi)d\xi \nonumber \\
&&\hskip 2cm+\int_{0}^{t}\int_{0}^{\infty}G(x,t;\xi,\tau)\left[V_\tau u_{\xi}(\xi,\tau)+\int_{0}^{\infty}dy u(y,\tau)p(y,\xi)\right]d\xi d\tau,
\end{eqnarray*}
where
\begin{eqnarray*}
\label{}
u(0,t)=2\int_{0}^{\infty}\int_{0}^{\infty}dy\rho(y,t)p(y,x)dx=:g(t),\ \varphi(\xi):=\rho_0^{\prime}(\xi).
\end{eqnarray*}
Then we can show that $\mathcal{F}$ is a contraction mapping for all sufficiently small $\eta>0$ and extend $\eta$ to $T$ as same as the proof of Proposition \ref{prop.1} so that  $\mathcal{F}$ has a unique fixed point $\displaystyle u\in C(\overline{D_T})$. Then $u$ is the unique solution which satisfies \eqref{intro.17}-\eqref{intro.19} (see Theorem 20.3.1 of \cite{Cannon}). 
\end{proof}
\vskip 0.3cm
Thus we can write $u$ as
\begin{eqnarray}
\label{123456789}
&&u(x,t)=-\int_{0}^{t}K_x(x,t;0,\tau)g(\tau)d\tau+\int_{0}^{\infty}G(x,t;\xi,0)\varphi(\xi)d\xi \nonumber \\
&&\hskip 1cm+\int_{0}^{t}\int_{0}^{\infty}G(x,t;\xi,\tau)\left[V_\tau u_{\xi}(\xi,\tau)+\int_{0}^{\infty}dy u(y,\tau)p(y,\xi)\right]d\xi d\tau,
\end{eqnarray}
where
\begin{eqnarray*}
\label{}
u(0,t)=2\int_{0}^{\infty}\int_{0}^{\infty}dy\rho(y,t)p(y,x)dx=:g(t),\ \varphi(\xi):=\rho_0^{\prime}(\xi).
\end{eqnarray*}
In addition, by differentiating \eqref{123456789} and integration by parts, we have
\begin{eqnarray}
\label{6789}
&& u_x(x,t)=-2\int_{0}^{t}K(x,t;0,\tau)g^{\prime}(\tau)d\tau+\int_{0}^{\infty}[K(x,t;\xi,0)+K(x,t;-\xi,0)    ]\varphi^{\prime}(\xi)d\xi\nonumber \\
&&\hskip 1.5cm+\int_{0}^{t}\int_{0}^{\infty}G_x(x,t;\xi,\tau)\left[V_\tau u_{\xi}(\xi,\tau)+\int_{0}^{\infty}dy u(y,\tau)p(y,\xi)\right]d\xi d\tau.
\end{eqnarray}

\medskip
We introduce
the following lemma from \cite{Cannon} which plays an essential role in our analysis.

\begin{lem}
\label{lem.1}
Let $\phi(t)$ satisfy
\begin{eqnarray}
\displaystyle 0\leq\phi(t)\leq\psi(t)+C_1\int_{0}^{t}\frac{\phi(\tau)}{\sqrt{t-\tau}}d\tau,\ 0\leq t\leq T,
\end{eqnarray}
where $C_1\geq 0$ and $\psi(t)$ is nonnegative and nondecreasing. Then
\begin{eqnarray}
0\leq \phi(t)\leq [1+2C_1\sqrt{t}]\psi(t)\exp\{\pi C_1^{2}t\}\leq C_2\psi(t)
\end{eqnarray}
with 
\begin{eqnarray}
C_2=[1+2C_1\sqrt{T}]\exp\{\pi C_1^{2}T\}.
\end{eqnarray}
\end{lem}
\begin{proof}
See Lemma 17.7.1 of \cite{Cannon}.
\end{proof}
 
Let $A>0$ and
   $$
\Sigma(A,T):=\Big\{V\in C([0,T]) : V_0=\frac{-\frac{1}{2}\rho_0^{\prime\prime}(0)+\int_{0}^{\infty}\int_{0}^{\infty}dy\rho_0^{\prime}(y)p(y,x)dx }{\rho_0^{\prime}(0)}, |V|_{\frac{1}{2}}\leq A \Big\},
   $$
where $\displaystyle|\cdot|_{\frac{1}{2}}$ is the H\"{o}lder seminorm with exponent $\dfrac{1}{2}$.

Let us denote by $\mathcal{S}$ a collection of continuous functions $C:[0,\infty)\times[0,\infty)\rightarrow\mathbb[0,\infty)$ such that for each $x\in[0,\infty)$, $C(x,\cdot)$ is increasing with respect to the second variable and $C(\cdot,0)>0$ is independent of the first variable.  \\

\begin{lem}
\label{lem.2}
Let $V\in\Sigma(A,T)$. For $(V,\rho)$ as in Proposition \ref{prop.1} and $(V,u)$ as in Proposition \ref{prop.2}, then $\displaystyle \lVert \rho \rVert_{\infty}\leq C_1(A,T)$, $\displaystyle\lVert \rho_x\rVert_{\infty}\leq C_2(A,T)$, $\displaystyle\sup_{0\leq t\leq T}\lVert \rho(\cdot,t)\rVert_{L^{1}}\leq C_3(A,T)$, $\displaystyle \lVert u \rVert_{\infty}\leq C_4(A,T)$, $\displaystyle\lVert u_x\rVert_{\infty}\leq C_5(A,T)$, $\displaystyle\sup_{0\leq t\leq T}\lVert u(\cdot,t)\rVert_{L^{1}}\leq C_6(A,T)$,  where $C_i\in\mathcal{S}$ for all $1\leq i \leq 6$.
\end{lem}
\begin{proof}
We can write $\rho$ as
\begin{eqnarray}
\label{145}
&&\displaystyle\rho(x,t)=\int_{0}^{\infty}G(x,t;\xi,0)\rho_0(\xi)d\xi \nonumber \\
&&\hspace{2cm}+\int_{0}^{t}\int_{0}^{\infty}G(x,t;\xi,\tau)\left[V_\tau\rho_\xi(\xi,\tau)+\int_{0}^{\infty}dy\rho(y,\tau)p(y,\xi)d\xi d\tau\right] \nonumber\\
&&\hspace{1.1cm}=\int_{0}^{\infty}G(x,t;\xi,0)\rho_0(\xi)d\xi-\int_{0}^{t}\int_{0}^{\infty}G_{\xi}(x,t;\xi,\tau)V_\tau\rho(\xi,\tau)d\xi d\tau\nonumber \\
&&\hspace{2cm}+\int_{0}^{t}\int_{0}^{\infty}G(x,t;\xi,\tau)\int_{0}^{\infty}dy\rho(y,\tau)p(y,\xi)d\xi d\tau. 
\end{eqnarray}
Then we have 
\begin{eqnarray}
\label{14141}
&&\displaystyle\sup_{x\geq 0}|\rho(x,t)|\leq 2 \lVert \rho_0\rVert_{\infty}+c_1\lVert V\rVert_{\infty}\int_{0}^{t}\frac{\sup_{\xi\geq 0}|\rho(\xi,\tau)| }{\sqrt{t-\tau}}d\tau  \nonumber\\
&&\hspace{2.5cm}+c_2\int_{0}^{t}\sup_{y\geq 0}|\rho(y,\tau)|d\tau.
\end{eqnarray}
Applying Lemma \ref{lem.1} on \eqref{14141}, we have
\begin{eqnarray*}
\label{}
&&\displaystyle\sup_{x\geq 0}|\rho(x,t)|\leq (1+2c_1\lVert V\rVert_{\infty}\sqrt{T})\exp\left\{\pi {c_1}^{2}\lVert V\rVert_{\infty}^{2}T\right\}\left(2 \lVert \rho_0\rVert_{\infty}+c_2\int_{0}^{t}\sup_{y\geq 0}|\rho(y,\tau)|d\tau\right).
\end{eqnarray*}
By Gronwall's lemma and $\lVert V\rVert_{\infty}\leq |V_0|+A\sqrt{T}$, we obtain for some $C_1\in\mathcal{S}$,
\begin{eqnarray}
\label{21}
\displaystyle \lVert \rho \rVert_{\infty}\leq C_1(A,T).
\end{eqnarray}

Also we can write $\rho_x$ as
\begin{eqnarray}
\label{678910}
&& \displaystyle \rho_x(x,t)=\int_{0}^{\infty}[K(x,t;\xi,0)+K(x,t;-\xi,0)]\varphi(\xi)d\xi\nonumber \\
&&\hskip 0.8cm+\int_{0}^{t}\int_{0}^{\infty}G_x(x,t;\xi,\tau)\left[V_\tau\rho_{\xi}(\xi,\tau)+\int_{0}^{\infty}dy \rho(y,\tau)p(y,\xi)\right]d\xi d\tau.
\end{eqnarray}
Then we have
\begin{eqnarray}
\label{141412}
\displaystyle\sup_{x\geq 0}|\rho_x(x,t)|\leq 2 \lVert \varphi\rVert_{\infty}+c_3\lVert V \rVert_{\infty}\int_{0}^{t}\frac{\sup_{\xi\geq 0}|\rho_\xi(\xi,\tau)| }{\sqrt{t-\tau}}d\tau+c_4\sqrt{T}\lVert \rho \rVert_{\infty}.
\end{eqnarray}
Using \eqref{21} and Lemma \ref{lem.1}, we also obtain for some $C_2\in\mathcal{S}$,
\begin{eqnarray}
\label{22}
\displaystyle \lVert \rho_x \rVert_{\infty}\leq C_2(A,T).
\end{eqnarray}

Taking absolute value on both sides of \eqref{145} and integrating them with respect to $x$, we have 
\begin{eqnarray}
\label{}
&&\displaystyle \lVert\rho(\cdot,t)\rVert_{L^{1}}\leq 2+c_5\lVert V\rVert_{\infty}\int_{0}^{t}\frac{\lVert \rho(\cdot,\tau) \rVert_{L^{1}} }{\sqrt{t-\tau}}d\tau+2\int_{0}^{t}\lVert \rho(\cdot,\tau) \rVert_{L^{1}}d\tau.
\end{eqnarray}
Similarly, we get for some $C_3\in\mathcal{S}$,
\begin{eqnarray}
\displaystyle \sup_{0\leq t\leq T}\lVert\rho(\cdot,t)\rVert_{L^{1}}\leq C_3(A,T).
\end{eqnarray}

Using \eqref{123456789} and integration by parts, we have
\begin{eqnarray}
\label{1414}
&&\displaystyle\sup_{x\geq 0}|u(x,t)|\leq \lVert g\rVert_{\infty}+2 \lVert \varphi\rVert_{\infty}+c_6\sqrt{T}\lVert V \rVert_{\infty}\lVert g\rVert_{\infty} \nonumber\\
&&\hspace{2cm}+c_7\lVert V \rVert_{\infty}\int_{0}^{t}\frac{\sup_{\xi\geq 0}|u(\xi,\tau)| }{\sqrt{t-\tau}}d\tau+c_8\int_{0}^{t}\sup_{y\geq 0}|u(y,\tau)|d\tau.
\end{eqnarray}
Using 
\begin{eqnarray}
\displaystyle \lVert g\rVert_{\infty}\leq  2\sup_{0\leq t\leq T} \lVert\rho(\cdot,t)\rVert_{L^{1}}\leq 2C_3(A,T)
\end{eqnarray}
and we apply both Gronwall's lemma and Lemma \ref{lem.1} on \eqref{1414}, we obtain for some $C_4\in\mathcal{S}$,
\begin{eqnarray}
\label{}
\displaystyle \lVert u\rVert_{\infty}\leq C_4(A,T).
\end{eqnarray}
By integration by parts, we have 
\begin{eqnarray}
\label{1313}
|g^{\prime}(t)|\leq c_9|\rho_x(0,t)|+(c_{10}|V_t|+c_{11})\lVert \rho(\cdot,t)\rVert_{L^{1}}.
\end{eqnarray}
Using \eqref{1313}, \eqref{6789} and  similar arguments as above, we finally get for some $C_5, C_6\in\mathcal{S}$,
\begin{eqnarray}
\displaystyle \lVert u_x\rVert_{\infty}\leq C_5(A,T),\ \ \sup_{0\leq t\leq T}\lVert u(\cdot,t)\rVert_{L^{1}}\leq C_6(A,T).
\end{eqnarray}
\end{proof}
\begin{lem}
\label{lem.3}
Let $V\in\Sigma(A,T)$. For $(V,u)$ as in Proposition \ref{prop.2}, then there exists $u_{xx}\in C(\overline{D_{T}} \setminus \mathbf{0})$ such that $\displaystyle \lVert u_{xx}\rVert_{\infty}\leq C_7(A,T)$ for some $C_7\in\mathcal{S}$.
\end{lem}
\begin{proof}
By differentiating \eqref{6789} with respect to spatial variable, let us define a mapping $\mathcal{F}:\mathscr{B_\eta}\rightarrow \mathscr{B_\eta}$ where $\displaystyle \mathscr{B_\eta}=\{v(x,t)\in C([0,\infty)\times[0,\eta]\setminus\mathbf{0}) : \lVert v \rVert_{\infty}<\infty\}$ as 
\begin{itemize}
\item if $x\ne 0$, $t\ne0$,
\begin{eqnarray}
&&\mathcal{F}v(x,t):=-2\int_{0}^{t}K_x(x,t;0,\tau)\left[g^{\prime}(\tau)-V_\tau u_x(0,\tau)-\int_{0}^{\infty}dy u(y,\tau)p(y,0) \right]d\tau\nonumber\\
&&\hskip 2cm+\int_{0}^{\infty}[K_x(x,t;\xi,0)+K_x(x,t;-\xi,0)    ]\varphi^{\prime}(\xi)dx \nonumber \\
&&\hskip 1.5cm+\int_{0}^{t}\int_{0}^{\infty}[K_x(x,t;\xi,\tau)+K_x(x,t;-\xi,\tau)]V_\tau v(\xi,\tau)d\xi d\tau\nonumber\\
&&-\int_{0}^{t}\int_{0}^{\infty}[K_x(x,t;\xi,\tau)+K_x(x,t;-\xi,\tau)]\int_{0}^{\infty}dy u(y,\tau)p_y(y,\xi) d\xi d\tau,
\end{eqnarray}
\item if $x=0$, $t\ne 0$, $\displaystyle\mathcal{F}v(0,t):=2\left[g^{\prime}(t)-V_t u_x(0,t)-\int_{0}^{\infty}dy u(y,t)p(y,0) \right]$,

\item if $x\ne 0$, $t=0$, $\mathcal{F}v(x,0):=\varphi^{\prime\prime}(x)$.
\end{itemize}

Then for $v_1, v_2\in \mathscr{B_\eta}$, we have the following estimate:
\begin{eqnarray*}
\lVert \mathcal{F}v_1-\mathcal{F}v_2 \rVert_{\infty}\leq c_1\sqrt{\eta}\lVert V \rVert_{\infty}\lVert v_1-v_2 \rVert_{\infty}.
\end{eqnarray*}
Thus for all sufficiently small $\eta>0$, $\mathcal{F}$ is a contraction mapping so that there is a unique fixed point $v^\eta$. We can extend $\eta$ to $T$ by a similar way of the proof of Proposition \ref{prop.1}. Let us say a unique $v\in C(\overline{D_{T}}\setminus\mathbf{0})$ such that for $x\ne 0$, $t\ne 0$,
\begin{eqnarray}
\label{123141}
&& v(x,t)=-2\int_{0}^{t}K_x(x,t;0,\tau)\left[g^{\prime}(\tau)-V_\tau u_x(0,\tau)-\int_{0}^{\infty}dy u(y,\tau)p(y,0) \right]d\tau \nonumber\\
&&\hskip 2cm+\int_{0}^{\infty}[K_x(x,t;\xi,0)+K_x(x,t;-\xi,0)    ]\varphi^{\prime}(\xi)d\xi\nonumber\\
&&\hskip 1.5cm+\int_{0}^{t}\int_{0}^{\infty}[K_x(x,t;\xi,\tau)+K_x(x,t;-\xi,\tau)]V_\tau v(\xi,\tau)d\xi d\tau\nonumber\\
&&-\int_{0}^{t}\int_{0}^{\infty}[K_x(x,t;\xi,\tau)+K_x(x,t;-\xi,\tau)]\int_{0}^{\infty}dy u(y,\tau)p_y(y,\xi) d\xi d\tau.
\end{eqnarray}
By integrating \eqref{123141} with respect to spatial variable on both sides and using integration by parts, we obtain for $x\ge0$, $t>0$,
\begin{eqnarray}
\label{1231415}
&& \int_{x}^{\infty}v(y,t)dy=2\int_{0}^{t}K(x,t;0,\tau)\left[g^{\prime}(\tau)-V_\tau u_x(0,\tau)-\int_{0}^{\infty}dy u(y,\tau)p(y,0) \right]d\tau \nonumber\\
&&\hskip 2cm+\int_{0}^{\infty}[-K(x,t;\xi,0)-K(x,t;-\xi,0)    ]\varphi^{\prime}(\xi)d\xi\nonumber\\
&&\hskip 1.5cm+\int_{0}^{t}\int_{0}^{\infty}[-K(x,t;\xi,\tau)-K(x,t;-\xi,\tau)]V_\tau v(\xi,\tau)d\xi d\tau\nonumber\\
&&\hskip 1cm+\int_{0}^{t}\int_{0}^{\infty}[K(x,t;\xi,\tau)+K(x,t;-\xi,\tau)]\int_{0}^{\infty}dy u(y,\tau)p_y(y,\xi) d\xi d\tau\nonumber\\
&&\hskip 2cm=2\int_{0}^{t}K(x,t;0,\tau)\left[g^{\prime}(\tau)-V_\tau u_x(0,\tau)\right]d\tau \nonumber \\
&&\hskip 2cm+\int_{0}^{\infty}[-K(x,t;\xi,0)-K(x,t;-\xi,0)    ]\varphi^{\prime}(\xi)d\xi\nonumber\\
&&+\int_{0}^{t}\int_{0}^{\infty} G_x(x,t;\xi,\tau)V_\tau \int_{\xi}^{\infty}v(y,\tau)dy d\xi d\tau-2\int_{0}^{t}K(x,t;0,\tau)V_\tau\int_{0}^{\infty}v(y,\tau)dy d\tau\nonumber\\
&&\hskip 1cm-\int_{0}^{t}\int_{0}^{\infty}G_x(x,t;\xi,\tau)\int_{0}^{\infty}dy u(y,\tau)p(y,\xi) d\xi d\tau.
\end{eqnarray}
Adding \eqref{6789} to \eqref{1231415}, we get
\begin{eqnarray}
\label{34141} 
&& u_x(x,t)+\int_{x}^{\infty}v(y,t)dy=2\int_{0}^{t}K(x,t;0,\tau)V_\tau\left\{-u_x(0,\tau)-\int_{0}^{\infty}v(y,\tau)dy\right\}d\tau \nonumber\\
&&\hskip 2cm+\int_{0}^{t}\int_{0}^{\infty}G_x(x,t;\xi,\tau)V_\tau\left\{u_{\xi}(\xi,\tau)+\int_{\xi}^{\infty}v(y,\tau)dy\right\}d\xi d\tau.
\end{eqnarray}

By letting $\displaystyle f(x,t):= u_x(x,t)+\int_{x}^{\infty}v(y,t)dy\in C(\overline{D_{T}})$, \eqref{34141} becomes
\begin{eqnarray*}
\label{2313131}
f(x,t)=-2\int_{0}^{t}K(x,t;0,\tau)V_\tau f(0,\tau)d\tau+\int_{0}^{t}\int_{0}^{\infty}G_x(x,t;\xi,\tau)V_\tau f(\xi,\tau)d\xi d\tau.
\end{eqnarray*}
By the uniqueness of the contraction mapping argument, $f$ should be identically $0$. Thus we conclude $v=u_{xx}\in C(\overline{D_{T}}\setminus\mathbf{0})$ and \eqref{123141}
becomes
\begin{eqnarray*}
\label{}
&& u_{xx}(x,t)=-2\int_{0}^{t}K_x(x,t;0,\tau)\left[g^{\prime}(\tau)-V_\tau u_x(0,\tau)-\int_{0}^{\infty}dy u(y,\tau)p(y,0) \right]d\tau\\
&&\hskip 2cm+\int_{0}^{\infty}G(x,t;\xi,0)\varphi^{\prime\prime}(\xi)d\xi\\
&&\hskip 2cm+\int_{0}^{t}\int_{0}^{\infty}[K_x(x,t;\xi,\tau)+K_x(x,t;-\xi,\tau)]V_\tau u_{xx}(\xi,\tau)d\xi d\tau\\
&&\hskip 2cm-\int_{0}^{t}\int_{0}^{\infty}[K_x(x,t;\xi,\tau)+K_x(x,t;-\xi,\tau)]\int_{0}^{\infty}dy u(y,\tau)p_y(y,\xi) d\xi d\tau.
\end{eqnarray*}
Then we have
\begin{eqnarray}
\label{3415}
&&\displaystyle \sup_{x\ge0}|u_{xx}(x,t)|\leq 2(\lVert g^{\prime} \rVert_{\infty}+\lVert V\rVert_{\infty}\lVert u_x\rVert_{\infty}+\lVert u\rVert_{\infty})+2\lVert \varphi^{\prime\prime}\rVert_{\infty} \nonumber\\
&&\hspace{3cm}+c_1\lVert V\rVert_{\infty}\int_{0}^{t}\frac{\sup_{\xi\ge0}|u_{xx}(\xi,\tau)| }{\sqrt{t-\tau}}d\tau+c_2\sqrt{T}\lVert u\rVert_{\infty}.
\end{eqnarray}
By applying Lemma \ref{lem.1} and Lemma \ref{lem.2} on \eqref{3415}, we deduce that for some $C_7\in\mathcal{S}$,
\begin{eqnarray*}
\displaystyle \lVert u_{xx}\rVert_{\infty}\leq C_7(A,T).
\end{eqnarray*}
\end{proof}

\begin{lem}
\label{lem.4}
Let $V\in\Sigma(A,T)$. For $(V,u)$ as in Proposition \ref{prop.2}, then $u_x(0,t)$ is H\"{o}lder continuous with exponent $\dfrac{1}{2}$ with $\displaystyle |u_x(0,\cdot)|_{\frac{1}{2}}\leq C_8(A,T)$ for some $C_8\in\mathcal{S}$.
\end{lem}
\begin{proof}
\begin{eqnarray*}
\label{}
&& u_x(0,t)=-2\int_{0}^{t}\frac{g^{\prime}(\tau)}{\sqrt{2\pi(t-\tau)}}d\tau+2\int_{0}^{\infty}\frac{1}{\sqrt{2\pi t}}\exp\left\{-\frac{\xi^2}{2t} \right\} \varphi^{\prime}(\xi)d\xi\\
&&+\int_{0}^{t}\int_{0}^{\infty}\frac{1}{\sqrt{2\pi(t-\tau)}}\frac{2\xi}{t-\tau}\exp\left\{-\frac{\xi^2}{2(t-\tau)}\right\}\left[V_\tau u_{\xi}(\xi,\tau)+\int_{0}^{\infty}dy u(y,\tau)p(y,\xi)\right]d\xi d\tau
\end{eqnarray*}
By change of variable, $\displaystyle w=\frac{\xi}{\sqrt{t}}$ and $\displaystyle z=\frac{\xi}{\sqrt{t-\tau}}$, we have
\begin{eqnarray*}
\label{}
&& u_x(0,t)=-2\int_{0}^{t}\frac{g^{\prime}(\tau)}{\sqrt{2\pi(t-\tau)}}d\tau+2\int_{0}^{\infty}\frac{1}{\sqrt{2\pi}}\exp\left\{-\frac{w^2}{2} \right\} \varphi^{\prime}(w\sqrt{t})dw\\
&&+\int_{0}^{t}\frac{1}{\sqrt{t-\tau}}\int_{0}^{\infty}\frac{2z}{\sqrt{2\pi}}\exp\left\{-\frac{z^2}{2}\right\}\left[V_\tau u_{\xi}(z\sqrt{t-\tau},\tau)+\int_{0}^{\infty}dy u(y,\tau)p(y,z\sqrt{t-\tau})\right]dz d\tau\\
&&=-2\int_{0}^{t}\frac{g^{\prime}(\tau)}{\sqrt{2\pi(t-\tau)}}d\tau+2\int_{0}^{\infty}\frac{1}{\sqrt{2\pi}}\exp\left\{-\frac{w^2}{2} \right\} \varphi^{\prime}(w\sqrt{t})dw+\int_{0}^{t}\frac{H(t,\tau)}{\sqrt{t-\tau}}d\tau,
\end{eqnarray*}
where $\displaystyle H(t,\tau)=\int_{0}^{\infty}\frac{2z}{\sqrt{2\pi}}\exp\left\{-\frac{z^2}{2}\right\}\left[V_\tau u_{\xi}(z\sqrt{t-\tau},\tau)+\int_{0}^{\infty}dy u(y,\tau)p(y,z\sqrt{t-\tau})\right]dz$.
\vskip 0.3cm
Moreover, we get for $t_1<t_2$,
\begin{eqnarray*}
&&\displaystyle\left|\int_{0}^{t_2}\frac{H(t_2,\tau)}{\sqrt{t_2-\tau}}d\tau-\int_{0}^{t_1}\frac{H(t_1,\tau)}{\sqrt{t_1-\tau}}d\tau\right|\\
&&\hspace{2cm}=\left|\int_{0}^{t_1}\frac{H(t_2,\tau)-H(t_1,\tau)}{\sqrt{t_2-\tau}}d\tau+\int_{0}^{t_1}H(t_1,\tau)\left[\frac{1}{\sqrt{t_2-\tau}}-\frac{1}{\sqrt{t_1-\tau}}\right]d\tau\right|\\
&&\leq c_1(\sqrt{t_2}-\sqrt{t_2-t_1})\sup_{\tau}|H(t_2,\tau)-H(t_1,\tau)|+c_2\sqrt{t_2-t_1}\sup_{\tau}|H(t_1,\tau)|
\end{eqnarray*}
Then by Lemma \ref{lem.2} and \ref{lem.3}, we obtain
\begin{eqnarray*}
|H(t_2,\tau)-H(t_1,\tau)|\leq C(A,T)\sqrt{t_2-t_1}
\end{eqnarray*}
for some $C\in\mathcal{S}$ and there is $\tilde{C}\in\mathcal{S}$ such that
\begin{eqnarray*}
\label{}
&& \displaystyle |u_x(0,t_2)-u_x(0,t_1)|\leq c_3\lVert g^{\prime}\rVert_{\infty}\sqrt{t_2-t_1}+c_4\sqrt{t_2-t_1}+\tilde{C}(A,T)\sqrt{t_2-t_1}\\
&&\hspace{3.5cm}\leq ( c_3\lVert g^{\prime}\rVert_{\infty}+c_4+\tilde{C}(A,T))\sqrt{t_2-t_1}.
\end{eqnarray*}
Thus we conclude for some $C_8\in\mathcal{S}$,
\begin{eqnarray*}
\displaystyle |u_x(0,\cdot)|_{\frac{1}{2}}\leq C_8(A,T).
\end{eqnarray*}
\end{proof}

\medskip

\begin{lem}
\label{lem.5}
There is $A>0$ such that for all sufficiently small $T>0$, 
$$\displaystyle Q[V](t)=\frac{-\frac{1}{2}u_x(0,t)+\int_{0}^{\infty}\int_{0}^{\infty}u(y,t)p(y,x)dydx}{u(0,t)},\ \ 0\leq t\leq T,$$
maps  $Q:\Sigma(A,T)\longrightarrow\Sigma(A,T)$.
  \end{lem}
  \begin{proof}
First of all, let $A>0$ and $T>0$ be arbitrary numbers. For $V\in\Sigma(A,T)$, let $(V,u)$ as in Proposition $\ref{prop.2}$. Since $u(0,t)=g(t)$ and $\displaystyle |f|_{\frac{1}{2}}\leq \sqrt{T}\lVert f^{\prime} \rVert_{\infty}$ for all $\displaystyle f\in C^{1}([0,T])$, so we have
 \begin{eqnarray*}
&&\displaystyle \left|\frac{-\frac{1}{2}u_x(0,t)+\int_{0}^{\infty}\int_{0}^{\infty}u(y,t)p(y,x)dydx}{u(0,t)}\right|_{\frac{1}{2}}\\
&&\hspace{1cm}\leq \frac{\lVert g\rVert_{\infty}\left(\frac{1}{2}|u_x(0,\cdot)|_{\frac{1}{2}}+\sqrt{T}\sup_{0\leq t\leq T}\left|\int_{0}^{\infty}\int_{0}^{\infty} u_t(y,t)p(y,x)dydx \right|\right)}{(\inf_{[0,T]}|g|)^{2}}\\
&&\hspace{2cm}+\frac{\sqrt{T}\lVert g^{\prime}\rVert_{\infty}\left(\frac{1}{2}\lVert u_x(0,\cdot)\rVert_{\infty}+\sup_{0\leq t\leq T}\lVert u(\cdot,t) \rVert_{L^{1}}\right)}{(\inf_{[0,T]}|g|)^{2}}
 \end{eqnarray*}
We also have
 \begin{eqnarray*}
\displaystyle \left|\int_{0}^{\infty}\int_{0}^{\infty} u_t(y,t)p(y,x)dydx \right|\leq c_1|u_x(0,t)|+(c_2|V_t|+c_3)|u(0,t)|+(c_4|V_t|+c_5)\lVert u(\cdot,t)\rVert_{L^{1}}
 \end{eqnarray*}
and
\begin{eqnarray*}
\inf_{[0,T]}|g|\geq |\varphi(0)|-T\lVert g^{\prime}\rVert_{\infty}\geq |\varphi(0)|-TC(A,T)
\end{eqnarray*}
for some $C\in\mathcal{S}$ by \eqref{1313}. Then by previous lemmas, for some $\tilde{C}\in\mathcal{S}$,
\begin{eqnarray*}
\displaystyle \left|\frac{-\frac{1}{2}u_x(0,t)+\int_{0}^{\infty}\int_{0}^{\infty}u(y,t)p(y,x)dydx}{u(0,t)}\right|_{\frac{1}{2}}\leq \frac{\tilde{C}(A,T)}{(|\varphi(0)|-TC(A,T))^{2}}.
\end{eqnarray*}
Let us choose $\displaystyle A>\frac{\tilde{C}(\cdot,0)}{|\varphi(0)|^{2}}$ and then for all sufficiently small $T>0$, $Q:\Sigma(A,T)\longrightarrow\Sigma(A,T)$ is well-defined.
\end{proof}
  
\begin{lem}
\label{lem.6}
The map $Q$ defined in  Lemma \ref{lem.5} is continuous on $\Sigma(A,T)$ with sup norm.
\end{lem}
\begin{proof}
Let $(V, \rho, u)$, $(\tilde{V}, \tilde{\rho}, \tilde{u})$ be two pairs as Proposition $\ref{prop.1}$ and $\ref{prop.2}$. Using abuse of notation, for each $C_i\in\mathcal{S}$, we will write simply $C_i$ instead of $C_i(A,T)$.\\
Using \eqref{3241} and taking a difference between $\rho$ and $\tilde{\rho}$, we have 
\begin{eqnarray*}
&&\displaystyle\sup_{x\geq0}|\rho(x,t)-\tilde{\rho}(x,t)|\leq c_1\sqrt{T}\lVert \rho\rVert_{\infty}\lVert V -\tilde{V}\rVert_{\infty}+c_2\lVert\tilde{V} \rVert_{\infty}\int_{0}^{t}\frac{\sup_{\xi\geq 0}|\rho(\xi,\tau)-\tilde{\rho}(\xi,\tau)|}{\sqrt{t-\tau}}d\tau\\
&&\hspace{4cm}+c_3\int_{0}^{t}\sup_{y\geq 0}|\rho(y,\tau)-\tilde{\rho}(y,\tau)|d\tau.
\end{eqnarray*}
Then similarly as before, by applying Gronwall's lemma and Lemma \ref{lem.1}, we obtain
\begin{eqnarray*}
\displaystyle \lVert\rho-\tilde{\rho}\rVert_{\infty}\leq C_1\lVert V -\tilde{V}\rVert_{\infty}.
\end{eqnarray*}
In addition,
\begin{eqnarray*}
&&\displaystyle\lVert\rho(\cdot,t)-\tilde{\rho}(\cdot,t)\rVert_{L^{1}}\leq c_4 \sqrt{T}\sup_{0\leq t\leq T}\lVert\rho(\cdot, t)\rVert_{L^{1}}\lVert V -\tilde{V}\rVert_{\infty}\\
&&\hspace{3.5cm}+c_5\lVert\tilde{V}\rVert_{\infty}\int_{0}^{t}\frac{\lVert\rho(\cdot,\tau)-\tilde{\rho}(\cdot,\tau)\rVert_{L^{1}}}{\sqrt{t-\tau}}d\tau+c_6\int_{0}^{t}\lVert\rho(\cdot,\tau)-\tilde{\rho}(\cdot,\tau)\rVert_{L^{1}}d\tau
\end{eqnarray*}
so that 
\begin{eqnarray*}
\displaystyle\sup_{0\leq t\leq T}\lVert\rho(\cdot,t)-\tilde{\rho}(\cdot,t)\rVert_{L^{1}}\leq C_2 \lVert V -\tilde{V}\rVert_{\infty}.
\end{eqnarray*}
Moreover, we also get
\begin{eqnarray*}
&&\displaystyle\sup_{x\geq0}|\rho_x(x,t)-\tilde{\rho}_x(x,t)|\leq c_7\sqrt{T}\lVert \rho_x\rVert_{\infty}\lVert V -\tilde{V}\rVert_{\infty}+c_{8}\lVert\tilde{V} \rVert_{\infty}\int_{0}^{t}\frac{\sup_{\xi\geq 0}|\rho_\xi(\xi,\tau)-\tilde{\rho}_\xi(\xi,\tau)|}{\sqrt{t-\tau}}d\tau\\
&&\hspace{4cm}+c_{9}\sqrt{T}\lVert \rho-\tilde{\rho} \rVert_{\infty}
\end{eqnarray*}
so that 
\begin{eqnarray*}
\displaystyle\sup_{x\geq0}|\rho_x(x,t)-\tilde{\rho}_x(x,t)|\leq C_3 \lVert V -\tilde{V}\rVert_{\infty}.
\end{eqnarray*}
By taking the difference of $u$ and $\tilde{u}$ written as \eqref{123456789}, we have
\begin{eqnarray*}
&&\displaystyle\sup_{x\geq0}|u(x,t)-\tilde{u}(x,t)|\leq \lVert g-\tilde{g} \rVert_{\infty}+c_{10}\sqrt{T}\lVert\tilde{V}\rVert_{\infty}\lVert g-\tilde{g} \rVert_{\infty}+c_{11}\sqrt{T}\lVert g \rVert_{\infty}\lVert V -\tilde{V}\rVert_{\infty}\\
&&\hspace{3.5cm}+c_{12}\sqrt{T}\lVert u\rVert_{\infty}\lVert V -\tilde{V}\rVert_{\infty}+c_{13}\lVert\tilde{V} \rVert_{\infty}\int_{0}^{t}\frac{\sup_{\xi\geq 0}|u(\xi,\tau)-\tilde{u}(\xi,\tau)|}{\sqrt{t-\tau}}d\tau\\
&&\hspace{3.5cm}+c_{14}\int_{0}^{t}\sup_{y\geq 0}|u(y,\tau)-\tilde{u}(y,\tau)|d\tau
\end{eqnarray*}
so that 
\begin{eqnarray*}
\displaystyle\sup_{x\geq0}|u(x,t)-\tilde{u}(x,t)|\leq C_4 \lVert V -\tilde{V}\rVert_{\infty}.
\end{eqnarray*}
Taking the difference between $\displaystyle g^{\prime}$ and $\displaystyle \tilde{g}^{\prime}$, we also have
\begin{eqnarray*}
&&\displaystyle|g^{\prime}(t)-\tilde{g}^{\prime}(t)|\leq c_{15}|\rho_x(0,t)-\tilde{\rho}_x(0,t)|+c_{16}|V_t|\lVert\rho(\cdot,t)-\tilde{\rho}(\cdot,t)\rVert_{L^{1}}+c_{17}\lVert\rho(\cdot,t)\rVert_{L^{1}}|V_t-V_t|\\
&&\hspace{2.5cm}+c_{18}\lVert\rho(\cdot,t)-\tilde{\rho}(\cdot,t)\rVert_{L^{1}}.
\end{eqnarray*}
so that 
\begin{eqnarray*}
\displaystyle \lVert g^{\prime}-\tilde{g}^{\prime}\rVert_{\infty}\leq C_5\lVert V -\tilde{V}\rVert_{\infty}.
\end{eqnarray*}
Taking the difference between $u_x$ and $\tilde{u}_x$ and using previous results, we deduce
\begin{eqnarray*}
&&\displaystyle\sup_{x\geq0}|u_x(x,t)-\tilde{u}_x(x,t)|\leq c_{19}\sqrt{T}\lVert g^{\prime}-\tilde{g}^{\prime} \rVert_{\infty}+c_{20}\sqrt{T}\lVert u_x\rVert_{\infty}\lVert V -\tilde{V}\rVert_{\infty}\\
&&\hspace{3cm}+c_{21}\lVert\tilde{V} \rVert_{\infty}\int_{0}^{t}\frac{\sup_{\xi\geq 0}|u_\xi(\xi,\tau)-\tilde{u}_\xi(\xi,\tau)|}{\sqrt{t-\tau}}d\tau+c_{22}\sqrt{T}\lVert u-\tilde{u} \rVert_{\infty}
\end{eqnarray*}
so that
\begin{eqnarray*}
\displaystyle \lVert u_x-\tilde{u}_x\rVert_{\infty}\leq C_6\lVert V -\tilde{V}\rVert_{\infty}.
\end{eqnarray*}
In a similar way as before and using the estimates above, we obtain
\begin{eqnarray*}
&&\displaystyle\lVert u(\cdot,t)-\tilde{u}(\cdot,t)\rVert_{L^{1}}\leq c_{23}\sqrt{T}\lVert g-\tilde{g}\rVert_{\infty}+c_{24}T\lVert g \rVert_{\infty}\lVert V -\tilde{V}\rVert_{\infty}+c_{25}T\lVert\tilde{V}\rVert_{\infty}\lVert g-\tilde{g}\rVert_{\infty}\\
&&\hspace{3.5cm}+c_{26}\sqrt{T}\lVert u\rVert_{L^{1}}\lVert V -\tilde{V}\rVert_{\infty}+c_{27}\lVert\tilde{V}\rVert_{\infty}\int_{0}^{t}\frac{\lVert u(\cdot,\tau)-\tilde{u}(\cdot,\tau)\rVert_{L^{1}}}{\sqrt{t-\tau}}d\tau\\
&&\hspace{3.5cm}+c_{28}\int_{0}^{t}\lVert u(\cdot,\tau)-\tilde{u}(\cdot,\tau)\rVert_{L^{1}}d\tau
\end{eqnarray*}
so that
\begin{eqnarray*}
\displaystyle\sup_{0\leq t\leq T}\lVert u(\cdot,t)-\tilde{u}(\cdot,t)\rVert_{L^{1}}\leq C_7\lVert V -\tilde{V}\rVert_{\infty}.
\end{eqnarray*}
Finally we conclude that
\begin{eqnarray*}
\displaystyle \lVert Q[V]-Q[\tilde{V}] \rVert_{\infty}\leq C_8\lVert V -\tilde{V}\rVert_{\infty}=C_8\lVert V-\tilde{V}\rVert_{\infty}.
\end{eqnarray*}
\end{proof}

\begin{proof} \hspace{-1.5mm}{\bf{of Theorem \ref{3424567}}}\\
$Q$ is continuous, then, since $\Sigma(A,T)$ is convex and compact we can apply the Schauder fixed point theorem to conclude that $Q$ has a fixed point. This completes the proof.
\end{proof}

\vskip 2cm

\section{Proof of Theorem \ref{thmintro.1}}
\label{sec.6}
Let $(V,\rho, u)$ in Theorem \ref{3424567} and let us define $\displaystyle\tilde{\rho}(x,t):=-\int_{x}^{\infty}u(y,t)dy$. We will show that  $\displaystyle\tilde{\rho}$ is the unique solution in Proposition \ref{prop.1}. Since $u$ satisfies
\begin{eqnarray*}
\frac{d}{dt}\left(\int_{0}^{\infty}u(y,t)dy\right)=-\frac{1}{2}u_x(0,t)-V_tu(0,t)+\int_{0}^{\infty}\int_{0}^{\infty}dzu(z,t)p(z,y)dy=0,
\end{eqnarray*}
thus we have $\displaystyle\tilde{\rho}(0,t)=-\int_{0}^{\infty}u(y,t)dy=-\int_{0}^{\infty}\rho_0^{\prime}(y)dy=0$.\\
Then $\tilde{\rho}$ satisfies
\begin{eqnarray*}
&&\displaystyle\tilde{\rho}_t(x,t)=-\int_{x}^{\infty}u_t(y,t)dy=\frac{1}{2}u_x(x,t)+V_tu(x,t)-\int_{x}^{\infty}\int_{0}^{\infty}dz\tilde{\rho}_z(z,t)p(z,y)dy\\
&&\hspace{1.5cm}=\frac{1}{2}\tilde{\rho}_{xx}(x,t)+V_t\tilde{\rho}_x(x,t)+\int_{x}^{\infty}\int_{0}^{\infty}dz\tilde{\rho}(z,t)p_z(z,y)dy\\
&&\hspace{1.5cm}=\frac{1}{2}\tilde{\rho}_{xx}(x,t)+V_t\tilde{\rho}_x(x,t)-\int_{0}^{\infty}dz\tilde{\rho}(z,t)\int_{x}^{\infty}p_y(z,y)dy\\
&&\hspace{1.5cm}=\frac{1}{2}\tilde{\rho}_{xx}(x,t)+V_t\tilde{\rho}_x(x,t)+\int_{0}^{\infty}dz\tilde{\rho}(z,t)p(z,x).
\end{eqnarray*}
Since $\displaystyle\tilde{\rho}(x,0)=\rho_0(x)$, by the uniqueness of Proposition \ref{prop.1}, we also get $\rho=\tilde{\rho}$ such that
\begin{eqnarray}
\label{21415}
\displaystyle\tilde{\rho}_x(0,t)=u(0,t)=2\int_{0}^{\infty}\int_{0}^{\infty}dy\rho(y,t)p(y,x)dx=2\int_{0}^{\infty}\int_{0}^{\infty}dy\tilde{\rho}(y,t)p(y,x)dx.
\end{eqnarray}
By \eqref{21415} and $\tilde{\rho}(0,t)=0$, we have $\displaystyle\frac{d}{dt}\left(\int_{0}^{\infty}\tilde{\rho}(y,t)dy\right)=0$ and finally deduce that
 \begin{eqnarray}
 \displaystyle \int_{0}^{\infty}\tilde{\rho}(x,t)dx=\int_{0}^{\infty}\rho_0(x)dx=1.
 \end{eqnarray}
This completes the proof.

\vskip 2cm

\section{Further results}
Let us try to apply our $\displaystyle C^{1}$-argument to the FBP of \cite{BBD17} as follows:
\[
(\star)\begin{cases}

     \displaystyle \rho_t(x,t)=\frac{1}{2}\rho_{xx}(x,t)+\rho,     & \quad \text{if } X_t<x,\ \ t> 0,\\

      \displaystyle \rho(X_t,t)=\alpha,     & \quad \text{if } t \ge 0,\\
     
    \displaystyle \rho_x(X_t,t)=\beta,  & \quad \text{if } t> 0.\\
\displaystyle \rho(x,0)=\rho_0(x),  & \quad \text{if } 0\leq x,\\
  \end{cases}
\]
where $\rho_0$ is specified later.\\
If $\alpha\neq 0$ and $\beta=0$, then $\displaystyle\frac{d}{dt}\rho(X_t,t)=V_t\rho_x(X_t,t)+\rho_t(X_t,t)=\frac{1}{2}\rho_{xx}(X_t,t)+\alpha=0$ so that, by change of variable $u:=\rho_x$, $(\star)$ becomes
\[
(\star\star)\begin{cases}

     \displaystyle u_t(x,t)=\frac{1}{2}u_{xx}(x,t)+u,     & \quad \text{if } X_t<x,\ \ t> 0,\\

      \displaystyle u(X_t,t)=0,     & \quad \text{if } t \ge 0,\\
     
    \displaystyle u_x(X_t,t)=-2\alpha,  & \quad \text{if } t> 0.\\
\displaystyle u(x,0)=\rho_0^{\prime}(x),  & \quad \text{if } 0\leq x,\\
  \end{cases}
\]
Again by change of variable $\displaystyle v:=-\frac{1}{2\alpha}u_x$, $(\star\star)$ becomes $(\star\star\star)$
\[
(\star\star\star)\begin{cases}

     \displaystyle v_t(x,t)=\frac{1}{2}v_{xx}(x,t)+v,     & \quad \text{if } X_t<x,\ \ t> 0,\\

      \displaystyle v(X_t,t)=1,     & \quad \text{if } t \ge 0,\\
     
     \displaystyle v(x,0)=-\frac{1}{2\alpha}\rho_0^{\prime\prime}(x),  & \quad \text{if } 0\leq x,\\
  
    \displaystyle  V_t=-\frac{1}{2}v_x(X_t,t),  & \quad \text{if } t> 0.\\
  \end{cases}
\]
To make each step valid, it needs that the value of the initial datum at $0$ is same as the boundary value which is $\rho_0(0)=\alpha$, $\rho_0^{\prime}(0)=0$, $\rho_0^{\prime\prime}(0)=-2\alpha$ and $\displaystyle\rho_{0}$ should be $\displaystyle C_c^{4}([0,\infty))$ to have $v_{xx}$ of $(\star\star\star)$ as in Lemma \ref{lem.3}.\\
Similarly as before, we can shift the boundary X to $0$ so that we have the following equivalent FBP:
\[
(\star\star\star^{\prime})\begin{cases}

     \displaystyle v_t(x,t)=\frac{1}{2}v_{xx}(x,t)+V_tv_x(x,t)+v(x,t)     & \quad \text{if } 0<x,\ \ t> 0,\\

      \displaystyle v(0,t)=1,     & \quad \text{if } t \ge 0,\\
     
     \displaystyle v(x,0)=-\frac{1}{2\alpha}\rho_0^{\prime\prime}(x),  & \quad \text{if } 0\leq x,\\
  
    \displaystyle  V_t=-\frac{1}{2}v_x(0,t),  & \quad \text{if } t> 0.\\
  \end{cases}
\]
By writing $v$ as
\begin{eqnarray}
\label{1234567}
&&v(x,t)=-\int_{0}^{t}K_x(x,t;0,\tau)d\tau+\int_{0}^{\infty}G(x,t;\xi,0)\psi(\xi)d\xi \nonumber \\
&&\hskip 2cm+\int_{0}^{t}\int_{0}^{\infty}G(x,t;\xi,\tau)\left[V_\tau v_{\xi}(\xi,\tau)+v(\xi,\tau)\right]d\xi d\tau,
\end{eqnarray}
where $\displaystyle \psi(\xi)=-\frac{1}{2\alpha}\rho_0^{\prime\prime}(\xi)$, we can repeat the same argument as previous sections such that there is a pair $(X,v)$ which satisfies $(\star\star\star^{\prime})$.\\
If $\alpha=0$, $\beta\neq 0$, it can be done similarly as $(\alpha\neq 0$, $\beta=0)$-case with the initial condition $\displaystyle\rho_{0}\in C_c^{3}([0,\infty))$ such that $\rho_0(0)=0$, $\rho_0^{\prime}(0)=\beta$.

{\bf Acknowledgments.}
I thank A. De Masi and E. Presutti for useful discussions.

\end{document}